\documentclass{amsart}
\usepackage{amsmath,amsthm,amssymb,xypic}
\usepackage[all]{xy}
\usepackage[bookmarksnumbered,plainpages,hypertex]{hyperref}
\usepackage{graphicx}
\usepackage{fancyhdr}
\usepackage{tikz}

\theoremstyle{plain}  

\newtheorem{thm}{Theorem}[section]
\newtheorem{pro}[thm]{Proposition}
\newtheorem{lem}[thm]{Lemma}   
\newtheorem{cor}[thm]{Corollary}

\newtheorem{convention}[thm]{Convention}
\newtheorem{construction}[thm]{Construction}

\newtheorem{remark}[thm]{Remark}

\theoremstyle{definition}

\newtheorem{definition}[thm]{Definition}

\theoremstyle{remark}

\DeclareMathOperator{\Spec}{Spec}

\DeclareMathOperator{\mult}{mult}
\DeclareMathOperator{\codim}{codim}

\DeclareMathOperator{\Ratn}{Rat^n}

\newcommand{\dJen}{de Jonqui\'eres\ }
\newcommand{\QED}{\ifhmode\unskip\nobreak\fi\quad {\rm Q.E.D.}} 

\newcommand\iso{\cong}

\newcommand{\f}{\varphi}

\newcommand{\F}{\mathbb{F}}

\newcommand{\I}{\mathcal{I}}

\renewcommand{\L}{\mathcal{L}}
\newcommand{\M}{\mathcal{M}}

\renewcommand{\O}{\mathcal{O}} 
\renewcommand{\P}{\mathbb{P}}

\newcommand{\R}{\mathcal{R}}

\newcommand{\rat}{\dasharrow}

\newcommand\Biro{{\rm Bir}^{\circ}}

\title{On the unirationality of 3-fold conic bundles}

\author{Massimiliano Mella}
\address{M. Mella\\
Dipartimento di Matematica e Informatica\\ 
Universit\`a di Ferrara\\
Via Machiavelli 35\\
44100 Ferrara Italia}
\email{mll@unife.it}
\date{March 2014}
\thanks{Partially supported by ``Geometria sulle Variet\`a Algebriche''(MIUR)}
\begin{document}
\maketitle
\section{Introduction}
A variety $X$ is unirational if it is dominated by a rational
variety. This was classically considered close to be rational and
 a long outstanding problem, after L\"uroth and Castelnuovo,
was to give examples of unirational non rational varieties. After
decades of struggle three different approaches gave almost
simultaneously the required examples, \cite{CG} \cite{IM}
\cite{AM}. Since then unirationality was considered old fashioned
 and it was gradually substituted by the more charming and powerful notion of
 rational connection. A variety is rationally connected if two
 general points can be joined by a rational curve. Rational connection
 is more suited
 for modern algebraic geometry and a great amount of results about
 rationally connected varieties come out of the theory of deformation of
 rational curves, see \cite{Ko}. Clearly this posed a new quest, still
 open, for
 rationally connected non unirational varieties. This paper is not
 going in this direction, but, following the stream opened in \cite{MM}, aims to show that the two notions can cooperate. The main result  is the following
 unirationality statement for 3-fold conic bundles, see definition \ref{def:cb}.
 
\begin{thm}\label{thm:conicbundle} Let $T$ be a 3-fold and $\pi:T\to W$ a standard
  conic bundle outside a codimension 2. Let $\Delta\subset W$ be the
  discriminant curve of $\pi$. Assume that there is a base
  point free pencil of rational curves $\Lambda$ in $W$, with $\Lambda\cdot \Delta\leq 7$. Assume that there is a very ample linear system $\M$ with $\M\cdot \Lambda=1$. Then $T$ is unirational.
 \end{thm}

To put Theorem \ref{thm:conicbundle} in the right perspective recall that 
Iskovskikh conjectured, \cite{Is}, that, except for a special well known case, a 3-fold standard conic bundle $\pi:T\to W$ with connected discriminant curve $\Delta$ is rational if and only if, up to birational isomorphism, there is a base point free pencil of rational curves $\Lambda$ with $\Lambda\cdot\Delta\leq 3$. It is quite easy to prove the sufficiency of this conjecture, \cite{Is}, while the necessity is very hard and Shokurov was able to prove it under the additional hypothesis that $W$ is either the plane or a minimal ruled surface $\F_e$, \cite{Sh}. 
The following is probably the most natural application of the previous theorem
 \begin{cor}\label{cor:main}
 Let $\pi:T\to W$ be a standard 3-fold conic bundle. Let $\Delta\subset W$ be the discriminant curve. Assume that one of the following is satisfied:
 \begin{itemize}
 \item[-] $W\iso\F_e$ and $\Delta\sim a C_0+ b F$, with $a\leq 7$,
\item[-] $W\iso\P^2$ and $\deg\Delta\leq 8$,
\item[-] $W\iso\P^2$,  $\deg\Delta= 9$, and $\Delta$ is singular.
 \end{itemize}
Then $T$ is unirational.
 \end{cor}

In particular Corollary \ref{cor:main}, thanks to Shokurov criteria, \cite{Sh}, produces infinitely many families of unirational 3-folds that are not rational.

 \begin{cor}\label{cor:main2}
 Let $\pi:T\to \F_e$ be a standard 3-fold conic bundle. Let $\Delta\subset W$ be the discriminant curve. Assume that $\Delta\sim a C_0+ b F$, with $3< a\leq 7$, then $T$ is unirational and not rational.
 \end{cor}

Let me briefly explain the main points of the proof.
The first step is to use  Enriques criterion and \cite{GHS} to 
reduce the unirationality of $T$ to the rational connection of some
subvariety in $\Ratn(S)$, where $S$ is a standard conic bundle
surface. Then this subvariety is proved to be birational to a
subvariety of $\Ratn(\P^2)$ that is seen to be a linear space.

The paper is organized as follows. The first section describes the reduction of
the unirationality statement to a statement about rational
connection of subvarieties of $\Ratn(S)$. The second section allows to substitute subvarieties in $\Ratn(S)$ with subvarieties in $\Ratn(\P^2)$. The third section proves the theorem.

{\sc acknowledgments}
   I would like to thank S\'andor Kov\'acs, this project started in
   Seattle and he prevented me many times from going in wrong
   directions.

\section{Notation and conventions}
I work over an algebraically closed field $k$ of characteristic zero.
Let me start with the following definitions.

\begin{definition}\label{def:cb}
  A standard conic bundle of dimension $n$ is a flat morphism $\pi:Z\to W$ with 
  \begin{itemize}
  \item[-] $n=\dim Z=\dim W+1$
\item[-] $Z$ and $W$ smooth
\item[-] $-K_Z$ $\pi$-ample
  \end{itemize}
For a standard conic bundle let $F\iso\P^1$ be the general fiber. I
say that $\pi:Z\to W$ is a standard conic bundle outside a codimension $a$ if there is
a dense open subset $U\subseteq W$, with $\codim U^c\geq a$,
such that $\pi_{|\pi^{-1}(U)}$ is a standard conic bundle.
\end{definition}

\begin{definition} Let $\F_e=\P_{\P^1}(\O\oplus\O(-e))$ be the Segre--Hirzebruch surface.
Fix a conic bundle structure  $\pi:S\to\P^1$, with general fiber $F$:
\begin{itemize}
\item[-] $C_0\subset S$ is the only section with negative self intersection if $e>0$ or a fixed section if $e=0$,
\item[-] $F_p:=\pi^{-1}(\pi(p))$  is the fiber through the point $p\in S$. 
\end{itemize}
\end{definition}

I am interested in studying rational curves on standard conic
bundles of dimension 2. I refer to \cite{Ko} for all the necessary definitions. To do this I adopt the following convention and definitions.
\begin{convention}Let $\pi:S\to\P^1$ be a standard conic bundle of dimension 2.
In the following $X, Y\subset S$ are such that:
\begin{itemize}
\item[-] $X, Y$ are disjoint zero dimensional reduced
  subschemes,
\item[-]$\sharp(\pi(X\cup Y))=\sharp(X\cup Y)$,
\item[-]for some $e\geq 0$, there is a birational map $\mu:S/\P^1\to\F_e/\P^1$ such that
  $\mu(X\cup Y)\subset\F_e\setminus C_0$, and $\mu^{-1}$ is an isomorphism in a neighborhood of $X$.
\end{itemize}\label{ref:xy}
\end{convention}


\begin{definition}\label{def:conicbundle}Let $S$ be a standard conic
  bundle of dimension 2. I define the algebraic sets
\begin{eqnarray*}
\R^{X}(Y;a,d)_S:=\overline{\{\mbox{irreducible rational curves in $|\I_{X^{2a}\cup
    Y^a}(-aK_{S}+d F)|$}\}}\subset
\\ \Ratn(S).
\end{eqnarray*}
\end{definition}
\begin{remark}\label{rem:CH} Let me stress that by construction the general element of any irreducible component of $\R^{X}(Y;a,d)_S$ is an irreducible rational curve.
 The varieties $\R^{X}(Y;a,d)_S$ can also be interpreted as a locally closed subset in $|-aK_{S}+d F|$. In this way it is a subvariety of the Severi variety $V(-aK_{S}+d F)$, see \cite{CH} for a modern account. I will switch between the two descriptions according to the convenience of the moment.
\end{remark}

The final ingredient are \dJen transformations. To introduce them let me recall the quasi projective variety $\Biro_d$.

Whenever $f_1,f_2,f_3\in k[x,y,z]_d$ are not all zero, let us denote by $[f_1 , f_2 , f_3]$ the equivalence class of the triplets $(f_1,f_2,f_3)$ with respect to the relation $(f_1,f_2,f_3)\sim(\lambda f_1,\lambda f_2, \lambda f_3)$, for $\lambda\in k^*$.
Consider $[f_1,f_2,f_3]$ as an element of $\P^{3N-1=3d(d+3)/2+2}$, where the homogeneous coordinates are all coefficients of the three polynomials $f_1,f_2,f_3$, up to multiplication by the same nonzero scalar for all of them.
Setting
\[
\omega\colon\P^2\rat\P^2, \qquad \omega([x,y,z]) = [f_1(x,y,z) , f_2(x,y,z) , f_3(x,y,z)],
\]
let us define, according to \cite{BCM},
\[
\Biro_d=\{  [ f_1 , f_2 , f_3 ] \mid \omega \text{ is birational and } \gcd(f_1,f_2,f_3)=1 \} \subset \P^{3N-1}.
\]

If $[f_1,f_2,f_3]\in\Biro_d$, let us identify it with the birational map $\omega$.

\begin{definition}\label{defgamma}
Let $\omega=[f_1,f_2,f_3]\in\Biro_d$, $d\geqslant2$.
Setting $B$ the linear span $\langle f_1,f_2,f_3 \rangle$ of the polynomials $f_1,f_2,f_3$ in $k^N$, the plane $\P(B)\subset\P^{N-1}$ is called the \emph{homaloidal net} associated to $\omega$ and I denote it by $\L_\omega$.

The general element of $\L_\omega$ defines an irreducible rational
plane curve of degree $d$ passing through some fixed points $p_1,\ldots,p_r$ in
$\P^2$, called \emph{set-theoretic base points} of $\L_\omega$, with certain multiplicities.
\end{definition}

\begin{definition}\label{DeJon}
The map $\omega\in\Biro_d$ is called a \emph{de Jonqui\`eres transformation} if there exists a base point of $\L_\omega$ with multiplicity $d-1$.
\end{definition}

\begin{remark}
The inverse of a \dJen transformation of degree $d$ is again a \dJen transformation of degree
$d$. Let $\omega:\P^2\rat \P^2$ be a \dJen map. Then $\L_\omega$ has $2d-2$ base points $\{q_1,\ldots,q_{2d-2}\}$,
eventually infinitely near, of multiplicity 1. 
 The simple base points are the residual intersection of a general element in $\L_\omega$ with a fixed curve of
degree $d-1$ having $p_0$ of multiplicity $d-2$. The latter curve is
contracted by $\omega$ to a smooth point, the base point of
multiplicity $(d-1)$ of the inverse.
\end{remark}

\begin{convention}\label{con:star}

Let $p_0:=[1,0,0]\subset\P^2$,
I will say that a zero dimensional reduced subset $Z\subset\P^2\setminus\{p_0\}$ satisfies condition $(\dag)$ if 
$$ \begin{array}{ll}
\hspace{-6,5cm} (\dag) & \langle q_i,p_0\rangle\not\ni q_j \mbox{\rm \ for any
             $i\neq j$}\\
 \end{array}
.$$
\end{convention}
\begin{definition}\label{def:centered} I say that a \dJen transformation $\omega:\P^2\rat\P^2$, of degree $d$,  is centered at $A$ if
  \begin{itemize}
\item[-] $A$ satisfies condition $(\dag)$ 
\item[-] $\L_\omega$ has multiplicity $d-1$ in $p_0$
\item[-] $|A|=2d-2$
\item[-]  $A$ is the set of simple base points of $\omega$.
  \end{itemize}
\end{definition}

 \section{From unirationality to rational connection}
\label{sec:uniratcon}
In this section the unirationality problem of
Theorem~\ref{thm:conicbundle} is translated into a statement on
rational connection of subvarieties in $\Ratn(S)$, where $S$ is a standard conic bundle of dimension 2.

Let $T$ be a 3-fold and $\pi:T\to W$ a conic bundle outside a codimension 2.
 Let
$\Delta\subset W$ be the discriminant curve, that is the curve that
describes the singular fibers.
Let $\Lambda$ be a base point free pencil of rational curves on $W$
and $f:W\to\P^1$ the morphism associated to $\Lambda$.
Consider the composition $\psi:=f\circ\pi$
$$\xymatrix{
T\ar[r]^\pi\ar@/_1pc/[rr]_{\psi}&
W\ar[r]^{f}&\P^1.&\\}
$$
Let $x\in\P^1$ be a general point and $S_x$ 
the fiber of $\psi$ over
$x$. Then $S_x$ is a standard conic bundle with
 $\delta$ reducible fibers.
Let $S_\eta$ be the generic fiber and
$$\overline{S}:=S_\eta\otimes_{k(x)}\Spec\overline{k(x)}$$
 the algebraic
closure. Then there is an $e\geq 0$ such that $\overline{S}$ is the
blow up in $\delta$ distinct points
$\{p_1,\ldots,p_\delta\}\subset\F_e\setminus C_0$ of the surface
$\F_e$. The curve $C_0\subset\overline{S}$ has self intersection
$C_0^2=-e$.
Let $\tilde{C}_0$ be a curve conjugate with $C_0$ over $k$, then $\tilde{C}_0^2=-e$. In particular the surface $S_x$ can be seen as the blow
up of $\F_e$ in $\delta$ points in $\F_e\setminus C_0$. 

I want to define subsets $X$ and $Y$
on the family of standard conic bundles $S_x$, keep in mind
Convention~\ref{ref:xy}. Assume that there is a very ample linear system $\M$ such that $\M\cdot \Lambda=1$.

 Let $m_i\in \M$ be a general element and 
 $D_{m_i}=\pi^{-1}(m_i)$ the corresponding surface in $T$. Then any section
 $\Sigma_{m_i}\subset D_{m_i}$ defines a point in $S_x$.
Fix $\{\Sigma_{m_1},\ldots,\Sigma_{m_s}\}$ general sections, and non negative integers $a$ and $d$.
Let 
$$Z_x=\cup_1^m\Sigma_{i}\cap S_x$$ 
and consider the subvarieties
\begin{eqnarray*}
 \R(\Sigma_{1},\ldots,\Sigma_{m};a,d):=\{([C],x)|
C\in\R^{Z_x}(\emptyset;a,d)_{S_x})\}\subset\\\Ratn(T/\P^1).
\end{eqnarray*}

Let 
\begin{equation}
  \label{eq:nu}
\nu:\R(\Sigma_{1},\ldots,\Sigma_{m};a,d)\to\P^1  
\end{equation}
 be the structure map.
 \begin{remark}
 We may interpret the generic fiber of $\nu$ as 
$\R^{Z_\eta}(\emptyset;a,d)_{S_\eta}$, where $Z_\eta$ is the set of generic points in $\cup_1^m\Sigma_{l_i}$.
 \end{remark}
The following proposition is the bridge I am looking for.
\begin{pro}\label{pro:reduction}
Let $\pi:T\to W$ be a 3-fold conic bundle outside a
codimension 2.  Assume that there is a base
  point free pencil of rational curves $\Lambda$ in $W$, and a very ample linear system $\M$ with $\M\cdot \Lambda=1$. Let
$$\R(\Sigma_{m_1},\ldots,\Sigma_{m_s};a,d)\subset \Ratn(T/\P^1)$$
 be as
above and assume that  there is a variety $R\subseteq \R(\Sigma_{m_1},\ldots,\Sigma_{m_s};a,d)$ such that for a general $x$: $R_x:=R\cap \nu^{-1}(x)$ is rationally connected and the general element in $R_x$ is irreducible. Then $T$ is unirational.
\end{pro}
\begin{proof} The proof is an easy consequence of Enriques criterion
  and the main result in \cite{GHS}. By Enriques criterion $T$ is
  unirational if and only if there is a rational surface $D\subset T$
  intersecting the general fiber of $\pi$. I am assuming that $R_x$ is rationally connected. Then by \cite{GHS} there is a section, say $\Gamma$, of
  the morphism $\nu$  in equation~(\ref{eq:nu}), through the general point of $R$. Then $\Gamma$ gives the desired rational surface.
\end{proof}
\begin{remark} I want to stress a further consequence of
  Proposition~\ref{pro:reduction}. Usually to prove (uni)rationality
  one has to work on non algebraically closed field to prove
  (uni)rationality results on the generic fiber of a morphism. Thanks to \cite{GHS} the unirationality problem I am interested in is
  reduced to study the rational connection of subvarieties in $\R^{Z}(\emptyset;a,d)_{S}$, for a standard conic
  bundle $S$ of dimension 2 defined over the algebraically closed field $k$.
\end{remark}
The next step is to substitute the standard conic bundle $S$ with $\P^2$.

\section{From conic bundles to $\P^2$}

Let $\pi:S\to\P^1$ be a standard conic bundle of dimension two. In
this section it is described  a variety in $\Ratn(\P^2)$ that is birational to the variety $\R^{Z}(\emptyset;a,d)_{S}$.

Let $\mu:S\to \F_e$ be the blow down to some $\F_e$ such that the
indeterminacy point of $\mu^{-1}$, say
$Y^\prime$, satisfies the convention~\ref{ref:xy}
$$\xymatrix{
S\ar[r]^{\tilde{\pi}}\ar[d]_{\mu}&\P^1\\
 \F_e\ar[ur]_{\pi}&}
$$
With this notations I have.
\begin{pro}
  \label{pro:conicbundle}
Fix $X, Y\subset\F_e$ with $Y^\prime\subseteq Y$ . Let $X_S=\mu^{-1}(X)$ and
 $Y_S=\mu^{-1}(Y\setminus Y^\prime)$ assume that
$\R^{X_S}(Y_S;a,d)_S$ is not empty and $\R^X(Y;a,d)_{\F_e}$ is
irreducible, then
$\R^X(Y;a,d)_{\F_e}$ is birational to $\R^{X_S}(Y_S;a,d)_S$.
\end{pro}
\begin{proof}
I have
$$-K_{S}=\mu^*(-K_{\F_e})-\sum_{y_i\in Y^\prime}E_{y_i}, $$
with $E_{y_i}$ the exceptional divisors over the point $y_i\in Y^\prime$, and 
$$(-aK_{{S}}+dF)\cdot E_{y_i}=a.$$
Let $[C]\in \R^{X_S}(Y_S;a,d)_S$ be a general point then $C$ is
irreducible and the above computation shows
that $[\mu(C)]\in \R^X(Y;a,d)_{\F_e}$.  This produces an injective map
$$\Theta:\R^{X_S}(Y_S;a,d)_S\rat\R^X(Y;a,d)_{\F_e}, $$
given by $\Theta([C])=[\mu(C)]$. By hypothesis $\R^X(Y;a,d)_{\F_e}$ is irreducible and the general element in $\Theta(\R^{X_S}(Y_S;a,d)_S)$  has multiplicity $a$ along $Y^\prime$. Hence  the  general
element in $\R^X(Y;a,d)_{\F_e}$ has multiplicity $a$ along $Y^\prime$.
This shows that the map 
$$\Psi:\R^X(Y;a,d)_{\F_e}\rat \R^{X_S}(Y_S;a,d)_S,$$ 
given by $\Psi([\Gamma])=[\mu^{-1}_*(\Gamma)]$ is defined on a dense open set and proves the claim.
\end{proof}

Proposition~\ref{pro:conicbundle} allows me to substitute $\R^{Z}(\emptyset;a, d)_{S}$ with $\R^X(Y^\prime;a, d)_{\F_e}$. The next step is to go to $\P^2$. For this consider the following construction.

\begin{construction}
  Let $X:=\{x_1,\ldots,x_{|e-1|+2k}\}\subset\F_e$. Fix a subset
  $X^\prime\subset X$ of $|e-1|$ points. Let $\f:\F_e\rat \F_1$ be the
  rational map obtained by $|e-1|$ elementary transformations centered
  in $X^\prime$, and $\epsilon:\F_1\to\P^2$ the contraction of the
  unique $(-1)$-curve to the point $p_0\in\P^2$. Let
  $f:\P^2\rat\P^2 $ be the \dJen transformation of degree $k+1$ centered in $(\epsilon\circ\f)(X\setminus X^\prime)$, see Definition~\ref{def:centered}. 
\begin{definition}
I call $\chi^X:=f\circ\epsilon\circ\f$ the composition.
\end{definition}
The map $f$ factors through the blow up $\epsilon$, therefore $\chi^X$ is independent of the choice of the subset $X^\prime\subset X$. 
\label{con:mapchiA}
\end{construction}

\begin{remark}\label{rem:mapchi}
Let $Y, X\subset\F_e$ be a subsets with $|X|=|e-1|+2k$. Then the map $\chi^X$ is composed with a \dJen map of degree $k$. In particular $(\epsilon\circ\f)(Y)$ satisfies condition $(\dag)$ and for $X$ general the points $\chi^X(Y)$ impose independent conditions on curves of degree $k$.
\end{remark}

An elementary transformation on $\F_e$ allows to switch to $\F_{|e-1|}$.
\begin{lem}\label{lem:elem}
Fix two subsets $X, Y\subset \F_e$, keep in mind Convention~\ref{ref:xy}. Let $\Phi_p:\F_e\rat \F_{|e-1|}$ be the elementary transformation
centered at $p\in X$. Define  $Y_p:=\Phi_p(Y)$ and
$X_p=\Phi_p(X\setminus\{p\})$. If 
$\R^{X_p}(Y_p;a,d-a)_{\F_{|e-1|}}$ is irreducible and
$\R^X(Y;a,d)_{\F_e}$ is not empty then $\R^X(Y;a,d)_{\F_e}$ is birational to $\R^{X_p}(Y_p;a,d-a)_{\F_{|e-1|}}$.
\end{lem}
\begin{proof}

Let $[\Gamma]\in \R^X(Y;a,d)_{\F_e}$ be any element, then 
$\Phi_p(\Gamma)\cdot F=2a$ and 
$$\Phi_p(\Gamma)\cdot C_0=\Gamma\cdot C_0=d-ae+2a=(d-a)-(|e-1|)a+2a.$$ 
This yields $[\Phi_p(\Gamma)]\sim -aK_{\F_{|e-1|}}+(d-a)F$. This produces an injective map
$$\Theta:\R^X(Y;a,d)_{\F_e}\rat \R^{X_p}(Y_p;a,d-a)_{\F_{|e-1|}}, $$
given by $\Theta([\Gamma])=[\Phi_p(\Gamma)]$.
 To conclude observe that the general curve in
$\R^{X}(Y;a,d)_{\F_{e}}$ is irreducible, therefore $\Theta([\Gamma])$ does
not contain $\Phi_p(F_p)$. Hence the general curve in
$\Theta(\R^X(Y;a,d)_{\F_e})$ does not contain $\Phi_p(F_p)$. Since  $\R^{X_p}(Y_p;a,d-a)_{\F_{|e-1|}}$ is irreducible then the  map 
$$\Psi:\R^{X_p}(Y_p;a,d-a)_{\F_{|e-1|}}\rat\R^X(Y;a,d)_{\F_e},$$ 
given by $\Psi([\Gamma])=[(\Phi_p^{-1})_*(\Gamma)]$ is well defined on a dense open subset and proves the claim.
\end{proof}

It is time to go to $\P^2$. 
\begin{definition}
\label{def:rp2} Let $A\subset\P^2\setminus\{p_0\}$ be a reduced
0 dimensional subscheme satisfying assumption $(\dag)$. I define the algebraic sets
\begin{eqnarray*}\R(A;a,d)_{\P^2}:=\overline{\{\mbox{irreducible rational curves in
  $|\I_{p_0^{d-2a}\cup A^a}(d)|$}}\subset\\\Ratn(\P^2).
\end{eqnarray*}
\end{definition}
\begin{remark}\label{rem:Testa}
When $d>1$ the varieties $\R(A;a,d)_{\P^2}$ are linear sections of a  Severi variety on $\mathbb{F}_1$.
\end{remark}

\begin{lem}\label{lem:elem1} Let $\mu:\F_1\to \P^2$ be the blow up of $\P^2$ in the
  point $p_0$.  Fix a subset $Y\subset \F_1$, remember Convention~\ref{ref:xy}.
If 
$\R(\mu(Y);a,d+3)_{\P^2}$ is irreducible and $\R^{\emptyset}(Y;a,d)_{\F_1}$ is not empty then $\R(\mu(Y);a,d+3)_{\P^2}$ is birational to $\R^{\emptyset}(Y;a,d)_{\F_1}$.
\end{lem}
\begin{proof}
  As in the previous Lemma note that $\mu$ is an isomorphism in a
  neighborhood of $Y$. Let $[\Gamma]\in \R^{\emptyset}(Y;a,d)_{\F_1}$ be
  an irreducible curve. Then  $\Gamma\cdot C_0=d+3-2a$ and $\mu(\Gamma)\sim\O(d+3)$
  has multiplicity $(d+3)-2a$ in $p_0$. This  yields  
$$[\mu(\Gamma)]\in\R(\mu(Y);a,d+3)_{\P^2}.$$
Then the general element in 
$\R(\mu(Y);a,d+3)_{\P^2}$ is a curve with multiplicity $d+3-2a$ in the
point $p_0$, therefore the strict transform is in $\R^{\emptyset}(Y;a,d)_{\F_1}$.
\end{proof}

These sum up to give the following proposition.

\begin{pro}\label{pro:cbtop2}
 Let $S$ be a standard conic bundle and $\mu:S\to\F_e$ the blow down
 to some $\F_e$ such that the indeterminacy points of $\mu^{-1}$, say
 $Y^\prime$, satisfy convention \ref{ref:xy}. Let $X,Y\subset\F_e\setminus C_0$ be
 subsets of general points with  $|X|=e-1+2k$, and $\chi^X:\F_e\rat\P^2$ the associated birational
 modification in Construction~\ref{con:mapchiA}. 
Assume that:
 \begin{itemize}
\item[-] $Y\supseteq Y^\prime$,
 \item[-]$\R(\chi^X(Y);a,d+3-a(e-1+2k))_{\P^2}$ is
 irreducible and not empty,
\item[-] the general element in
  $\R(\chi^X(Y);a,d+3-a(e-1+2k))_{\P^2}$ has multiplicity $d+3-2a-a(e-1+2k)$ in
 $p_0$, and does not contain the indeterminacy locus of
 $(\chi^X)^{-1}$,
\item[-] the general element in
  $\R(\chi^X(Y);d+3-a(e-1+2k))_{\P^2}$
 has multiplicity $a$ along $\chi^X(Y)$.
 \end{itemize}
Then the variety $\R^{\mu^{-1}(X)}(\mu^{-1}(Y\setminus Y^\prime);a,d)_S$ is birational to $\R(\chi^X(Y);a,d+3-a(e-1+2k))_{\P^2}$.
\end{pro}
\begin{proof}
First observe that $\R^{\mu^{-1}(X)}(\mu^{-1}(Y\setminus Y^\prime);a,d)_S$ is
not empty. I am assuming that  the general
element in
$$\R(\chi^X(Y);a,d+3-a(e-1+2k))_{\P^2}$$
does not contain $(\chi^X)^{-1}\setminus\{p_0\}$,  has multiplicity $d+3-2a-a(e-1+2k)$ in
 $p_0$, and has multiplicity $a$ in $\chi^X(Y)$.  Let
$$[\Gamma]\in\R(\chi^X(Y);a,d+3-a(e-1+2k))_{\P^2}$$
be a general point, then
$[(\chi^X)^{-1}(\Gamma)]\in \R^{\mu^{-1}(X)}(\mu^{-1}(Y\setminus Y^\prime);a,d)_S$. Then I can apply backwards
recursively Lemma~\ref{lem:elem1}, Lemma~\ref{lem:elem}, and
Lemma~\ref{pro:conicbundle} to get the desired conclusion.
\end{proof}

\section{Proof  of the main result}
The first step in the proof is  to produce a  rationally connected subvariety in  some $\R(Y;a,d)_{\P^2}$.
 
\begin{pro}\label{pro:singular}
 Let $Y\subset\P^2\setminus\{p_0\}$ be a set of points satisfying  assumption
 $(\dag)$, and $f$ a general \dJen transformation of degree $d\geq
 4$ with $\mult_{p_0}\L_f=d-1$. Assume that $|Y|= 7$, then
 $\R(f(Y);4,11)_{\P^2}\iso \P^1$, moreover the general curve in   $\R(f(Y);4,11)_{\P^2}$ is irreducible with 4-tuple points along $f(Y)$, and a 3-ple point in $p_0$.
\end{pro}
\begin{proof} A dimension count shows that $\R(Y;4,11)_{\P^2}$
  is expected to be a linear space of dimension 1.
Let
$\{x_1,\ldots,x_7\}$ be the points in $Y$, then I may assume that
$f(Y)$ has neither three collinear points nor 6 points on a conic, see Remark~\ref{rem:mapchi}.
Let $\omega_1:\P^2\rat\P^2$
be the composition of the standard Cremona transformations centered in
$\{x_1,x_2,x_3\}$ and $\{x_4,x_5,x_6\}$. Let $\{y_1,y_2,y_3\}$ and
$\{z_1,z_2,z_3\}$ be the exceptional points of $\omega^{-1}$. The
general choice of $f$ allows me to assume that
$\{y_1,y_2,y_3,z_1,z_2,z_3,x_7\}$ have neither three collinear points
nor 6 points on a conic. Let
$\omega_1$ be the composition of the standard Cremona transformations
centered in $\{x_7,y_1,y_2\}$ and $\{y_3,p_0,z_1\}$. Let
$\{w_1,w_2,w_3\}$ and $\{t_1,t_2,t_3\}$ be the exceptional
points of $\omega_1^{-1}$. Then again I may assume that $\{z_2,z_3,w_1,w_2,w_3,t_1,t_2,t_3\}$
are distinct points without three collinear points, 6 points on a
conic, and they are not contained by a rational cubic curve. 
Let $\Lambda$ be the pencil of quartic curves singular
in $\{z_2,z_3,w_1\}$ and passing through $\{w_2,w_3,t_1,t_2\}$. Then a
direct and straightforward computation shows that
the strict transform linear system
$(\omega_1\circ\omega)^{-1}_*(\Lambda)$ is a pencil of irreducible
curves of degree 11 with 4-tuple points in $Y$ and a triple point in $p_0$.
The following table summarizes the computation. 
$$\begin{array}{|r|cccccccc|}
\hline\deg&z_1&z_2&z_3&w_1&w_2&w_3&t_1&t_2\\
\hline
4 &{\bf 0}&2&2&2&{\bf 1}&{\bf 1}&1&1\\\hline
6 &2&{\bf 2}&2&2&3&3&{\bf 1}&{\bf 1}\\\hline
8 &{\bf 2}&4&{\bf 2}&{\bf 2}&3&3&3&3\\\hline
10&4&4&4&4&{\bf 3}&{\bf 3}&{\bf 3}&3\\\hline
11&4&4&4&4&4&4&4&3\\  \hline
\end{array}
$$
The columns indicate the degree of the curve and the corresponding  multiplicities in the points. To pass from one row to the next apply the standard Cremona transformation centered in the bold points.
\end{proof}
\begin{remark} It is reasonable to expect that the linear system in 
Proposition~\ref{pro:singular} is the best possible. In other words, in the notation of section~\ref{sec:uniratcon}, I
expect that  there are no subvarieties $$R\subset\R(\Sigma_{m_1},\ldots,\Sigma_{m_s};a,d)\subset \Ratn(T/\P^1)$$ 
such that $R_x$ is a linear space of positive dimension and the general element is irreducible as
soon as $|Y|>7$.
\end{remark}

\begin{proof}[Proof of Theorem~\ref{thm:conicbundle}]
In the
notation of section~\ref{sec:uniratcon} let $S_x$ be a general fiber
of the morphism $\psi:T\to\P^1$. Then $S_x$ is a standard conic
bundle of dimension 2 and I may assume that it is the blow up of
$\F_e$ along a subset $Y^{\prime}$, of length $\delta\leq 7$, that satisfies
the Convention~\ref{ref:xy}. Let $\mu_x:S_x\to\F_e$ be any such map. Fix
$s=e-1+8$ general sections $\{\Sigma_1,\ldots,\Sigma_{s}\}$ of the
morphism $\psi$, and 
$$Z_x=\cup_1^s\Sigma_{i}\cap S_x.$$ 
Fix $7-\delta$ more general sections
$\{\Sigma_{s+1},\ldots,\Sigma_{s+7-\delta}\}$ and let $Y_x=\cup_1^{7-\delta}\Sigma_{s+i}\cap S_x.$.

By Construction~\ref{con:mapchiA} the map $\chi^{Z_x}\circ\mu_x$ is composed with a general \dJen transformation of degree 4. Then by Proposition~\ref{pro:singular} I have:
\begin{itemize}
\item[i)] $\R((\chi^{Z_x}\circ\mu_x)(Y^\prime\cup Y_x);4,11)_{\P^2}\iso\P^1$,
\item[ii)] the general curve in   $\R((\chi^{Z_x}\circ\mu_x)(Y^\prime\cup Y_x);4,11)_{\P^2}$ is irreducible with 4-tuple points along $(\chi^{Z_x}\circ\mu_x)(Y^\prime\cup Y_x)$, and a 3-ple point in $p_0$.
\end{itemize}
Then I may apply  Proposition~\ref{pro:cbtop2} to get
$$\R^{Z_x}(Y_x;4,36+4e)_{S_x}\iso\P^1.$$
Together with ii), keep also in mind Remark~\ref{rem:CH}, this is enough to conclude  by Proposition~\ref{pro:reduction}.
\end{proof}
\begin{proof}[Proof of Corollary~\ref{cor:main}]
Assume first that $W\iso\F_e$ with $f:W\to\P^1$ a conic bundle structure. The linear systems $\Lambda=f^*\O(1)$ and $\M\sim C_0+(e+1)f$ satisfies the assumption of Theorem~\ref{thm:conicbundle} and allows to conclude.

If $W\iso\P^2$ fix a point $z\in \Delta$. Since $\pi$ is a standard conic bundle then $\Delta$ has at most ordinary double points. Let $\mu:\F_1\to\P^2$ be the blow up of $z$. It is known, \cite[Proposition 2.4]{Sa}, that there is a standard conic bundle $\pi_1:X_1\to\F_1$  such that the following diagram  is commutative
$$\xymatrix{
X_1\ar@{.>}[r]^{\epsilon}\ar[d]_{\pi_1}&X\ar[d]^{\pi}\\
 \F_1\ar[r]^\mu&\P^2\\}.
$$
Moreover the discriminant curve of $\pi_1$ is $\Delta_1=\mu^{-1}(\Delta)\setminus C_0$, \cite[Corollary 2.5]{Sa}. That is  $\Delta_1\sim aC_0+bF$ with $a=
 \deg\Delta-\mult_{z}\Delta$. This is enough to conclude again by Theorem~\ref{thm:conicbundle}.
\end{proof}


\begin{thebibliography}{Comes}
\bibitem[AM]{AM} Artin, M.; Mumford, D.
{\it Some elementary examples of unirational varieties which are not rational} 
Proc. London Math. Soc. (3) {\bf 25} (1972), 75–95. 
\bibitem[CG]{CG} Clemens, C. Herbert; Griffiths, Phillip A.
{\it The intermediate Jacobian of the cubic threefold}
Ann. of Math. (2) {\bf 95} (1972), 281–356. 
\bibitem[BCM]{BCM} Bisi, Cinzia; Calabri, Alberto; Mella, Massimiliano {\it On Plane Cremona Transformations of Fixed Degree} Journal of Geometric Analysis  to appear DOI 10.1007/s12220-013-9459-9
\bibitem[CH]{CH} Caporaso, Lucia; Harris, Joe {\it Parameter spaces for curves on surfaces and enumeration of rational curves} Compositio Math. {\bf 113} (1998), no. 2, 155–208.
\bibitem[GHS]{GHS} Graber, Tom; Harris, Joe; Starr, Jason {\it Families of rationally connected varieties}. J. Amer. Math. Soc. {\bf 16} (2003), no. 1, 57–67.
\bibitem[Is]{Is}  Iskovskikh, V. A. {\it On the rationality problem for conic bundles}, Duke Math. J. {\bf 54} (1987), 271-294.
\bibitem[IM]{IM}  Iskovskikh, V. A.; Manin, Ju. I. {\it Three-dimensional quartics and counterexamples to the L\"uroth problem} (Russian) Mat. Sb. (N.S.) {\bf 86}(128) (1971), 140–166.
\bibitem[Ko]{Ko}  Koll\'ar, J\'anos {\it Rational curves on algebraic varieties} Ergebnisse der Mathematik und ihrer Grenzgebiete. 3. Folge. A Series of Modern Surveys in Mathematics [Results in Mathematics and Related Areas. 3rd Series. A Series of Modern Surveys in Mathematics], {\bf 32} Springer-Verlag, Berlin, 1996. viii+320 pp.
\bibitem[MM]{MM} Massarenti, Alex; Mella, Massimiliano {\it Birational aspects of the geometry of varieties of sums of powers} Adv. Math. {\bf 243} (2013), 187–202.
\bibitem[Sa]{Sa} Sarkisov, V.G. {\it Birational automorphisms of conic bundles} Math. USSR Izvestija Vol. {\bf 17} (1981) 177-202.
\bibitem[Sh]{Sh} V.V. Shokurov, {\it Prym varieties: theory and applications}, Izv. Akad. Nauk SSSR {\bf 47} (1983), 785-855, translated in Math. USSR-Izv. {\bf 23} (1984), 83-147.
\end{thebibliography}
\end{document}